\documentclass[12pt,a4paper]{article}
\usepackage[a4paper,margin=2cm]{geometry}

\RequirePackage{etex} 
\usepackage{amsmath}
\usepackage{amsthm}
\usepackage{subcaption}
\usepackage{amssymb}
\usepackage{enumitem}
\usepackage{graphicx}
\usepackage{nccmath}
\usepackage{setspace}
\usepackage[colorlinks=true, allcolors=blue]{hyperref}
\usepackage{enumitem}
\usepackage{multirow}
\usepackage{float}
\usepackage[dvipsnames]{xcolor}
\usepackage{tikz}
\usetikzlibrary{arrows.meta}
\usetikzlibrary{calc}

\newtheorem{theorem}{Theorem}

\newtheorem{lemma}[theorem]{Lemma}

\newtheorem{note}{Note}

\usepackage{thmtools}

\theoremstyle{definition}
\newtheorem{example}[theorem]{Example}

\def \mod#1{{\:({\rm mod}\ #1)}}

\def \geq {\geqslant}

\let\oldproofname=\proofname
\renewcommand{\proofname}{\textup{\textbf{\oldproofname}}}

\title{Chain--collider--fork Decompositions of Transitive Tournament}

\author{
Ajani De Vas Gunasekara\\[0.5em]
{\small School of Arts and Sciences, The University of Notre Dame Australia}\\
{\small Sydney NSW 2007, Australia}\\
{\small ajani.de.vas.gunasekara@nd.edu.au}}

\date{}

\begin{document}

\maketitle

\begin{abstract}
A transitive tournament is an acyclic orientation of a complete graph. 
We study decompositions and packings of the transitive tournament \(TT_n\) into connected two-arc motifs. 
The three motifs considered are chains, colliders, and forks, which are also fundamental local configurations in directed acyclic graphs. 
We first construct decompositions of \(TT_n\) into mixtures of these motifs whenever such decompositions exist. 
We then consider the corresponding pure packing problem for each individual motif. 
For \(H\) equal to a chain, a collider, or a fork, we determine the maximum number of arc-disjoint copies of \(H\) in \(TT_n\). 
These results give a precise extremal description of two-arc motif packings in transitive tournaments and suggest further questions on motif decompositions in broader classes of directed acyclic graphs.
\end{abstract}

\section{Introduction}

An \textit{undirected graph} or simply a \textit{graph} $G$ is a pair $(V,E)$ where $V$ is the set of vertices and $E$ is the set of edges consisting of unordered pairs of vertices. We denote the vertex set and the edge set of $G$ by $V(G)$ and $E(G)$. A \textit{directed graph} or a \textit{digraph} $D$ consists of a non-empty finite set of \textit{vertices} $V(D)$ and a finite set $A(D)$ of ordered pairs of distinct vertices called \textit{arcs}; often denoted as $D = (V,A)$. For an arc $(u,v) \in A(D)$, $u$ is called its tail, and $v$ is called its head. Furthermore, $u$ dominates $v$ or $v$ is dominated by $u$ and denoted by $ u \rightarrow v$. A digraph $H$ is a \textit{subdigraph} of a digraph $D$ if $V(H) \subseteq V(D), A(H) \subseteq A(D),$ and every arc in $A(H)$ has both end-vertices in $V(H)$.

An \textit{orientation} of an undirected graph is the assignment of a specific direction to each edge, transforming it into a digraph. A \textit{tournament} is a digraph with exactly one edge between each two vertices, in one of the two possible directions. Equivalently, a tournament is an orientation of the complete graph. A digraph $D$ is \emph{transitive} if for every three distinct vertices $u,v,w \in V(D)$, $(u,v),(v,w) \in A(D)$ implies that $(u,w) \in A(D)$. 
Although there are many non-isomorphic tournaments on \(n\) vertices, there is a unique transitive tournament. It is obtained by labelling the vertices \(v_1,\dots,v_n\) and orienting each edge from \(v_i\) to \(v_j\) whenever \(i<j\) (see \cite{BangJensenGutin2009}). 
We denote the transitive tournament on \(n\) vertices, or equivalently the transitive tournament of order \(n\), by \(TT_n\). The vertex labelling described above is referred to as a \textit{topological ordering} or \textit{acyclic ordering} of \(TT_n\). Throughout, we represent \(TT_n\) as in Figure~\ref{fig:topo order}.

\begin{figure}[ht]
    \centering

\begin{tikzpicture}[scale=1.3, transform shape, >=Stealth, every node/.style={font=\small}]
  \node[circle, fill=black, inner sep=1.3pt, label=below:$v_1$]   (v1)  at (0,0)   {};
  \node[circle, fill=black, inner sep=1.3pt, label=below:$v_2$]   (v2)  at (1.4,0) {};
  \node[circle, fill=black, inner sep=1.3pt, label=below:$v_{n-2}$] (vn2) at (5.2,0) {};
  \node[circle, fill=black, inner sep=1.3pt, label=below:$v_{n-1}$] (vn1) at (6.4,0) {};
  \node[circle, fill=black, inner sep=1.3pt, label=below:$v_n$]   (vn)  at (7.6,0) {};

  \draw [->] (v1) to (v2);
  \draw[dotted] (v2) -- (vn2);
  \draw [->] (vn2) to (vn1);
  \draw [->] (vn1) to (vn);


  \draw[->] (v1) to[bend left=35] (vn2);
  \draw[->] (v1) to[bend left=45] (vn1);
  \draw[->] (v1) to[bend left=55] (vn);

  \draw[->] (v2) to[bend left=25] (vn2);
  \draw[->] (v2) to[bend left=35] (vn1);
  \draw[->] (v2) to[bend left=45] (vn);

  \draw[->] (vn2) to[bend left=35] (vn);
\end{tikzpicture}

    \caption{Topological ordering of a transitive tournament of order $n$}
    \label{fig:topo order}
\end{figure}
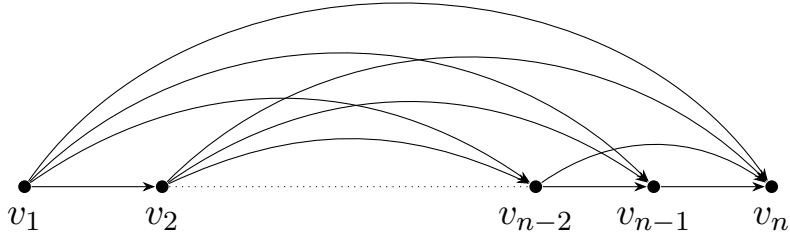

Recent years have witnessed considerable activity in the study of graph and hypergraph decompositions. A fundamental question in this area is whether the edge set of a graph from a given class, such as complete graphs, hypergraphs, or directed graphs, can be partitioned into subgraphs satisfying a prescribed property. For digraphs \(D\) and \(H\), we say that \(D\) is \(H\)-decomposable if \(A(D)\) admits a partition into pairwise disjoint sets, each inducing a subdigraph isomorphic to \(H\). An obvious necessary condition for the existence of such a decomposition is that \(|A(H)|\) divides \(|A(D)|\).
More generally, if a family of pairwise arc-disjoint copies of \(H\) in \(D\) covers only a proper subset of \(A(D)\), then this family is called an \textit{\(H\)-packing} of \(D\). The \textit{$H$-packing number} of $D$ is the maximum size of a $H$-packing of $D$.

In this paper, we study decompositions of transitive tournaments into connected two-arc digraphs. Ignoring orientation, each such digraph is isomorphic to a path of length two. We consider decompositions involving various combinations of the different isomorphism types of these digraphs and determine the maximum number of copies of each type that can be attained. Equivalently, we determine the packing number of each isomorphism class of \(TT_n\).

Path decompositions of various types of graphs have been studied extensively, including complete graphs \cite{Tarsi1983}, bipartite graphs \cite{ChuFanZhou2021} and graph products \cite{DeVasDevillers2024}.
The problem of decomposing digraphs into directed paths was first explored by Alspach and Pullman in 1974 \cite{AlspachPullman1974}. They established bounds on the minimum number of paths necessary for such decompositions. A conjecture of Alspach, Mason, and Pullman asks for the minimum number of paths needed in a path decomposition of a general tournament $T$. There is a natural lower bound for this number in terms of the degree sequence of $T$ and it is conjectured that this bound is correct for tournaments of even order \cite{AlspachMasonPullman1976}; \cite{LoPatelSkokanTalbot2020} and \cite{GiraoGranetKuhnLoOsthus2023} have made progress on the conjecture. 

In \cite{Gorlichetal2007} and \cite{Gorlichetal2007II}, the authors characterised the digraphs \(H\) with at most four edges such that \(TT_n\) admits an \(H\)-decomposition. The connected case is treated in \cite{Gorlichetal2007}, whereas \cite{Gorlichetal2007II} deals with disconnected digraphs $H$. The disconnected case is of particular relevance here since some of the connected components considered are connected digraphs with two edges. Further related work was carried out in \cite{SaliSimonyiTardos2021}, where the authors studied decompositions of transitive tournaments into a prescribed number of isomorphic digraphs. Moreover, \cite{WangWuAn2017} provides a characterisation of tournaments that admit decompositions into two-arc digraphs isomorphic to a directed path.

The remainder of the paper is organised as follows. Section~\ref{S:Section 2} lays the foundations for the paper by introducing further key terms and definitions. Section~\ref{S:sec3} is devoted to establishing the main tool used in proving the results throughout the paper. Sections \ref{S:sec4}--\ref{S:sec6} examine the different isomorphism classes of connected two-arc digraphs in detail and establish the corresponding results. Finally, Section~\ref{S:sec7} concludes the paper with a discussion of common properties of these decompositions and possible directions for future research.

\section{Preliminaries}
\label{S:Section 2}

This section will establish the foundation for our work. For a vertex $v\in V(D)$, the \emph{in-degree} of $v$ is the number of arcs in $A(D)$ whose head is $v$, and the \emph{out-degree} of $v$ is the number of arcs in $A(D)$ whose tail is $v$. These are denoted by $\deg_D^{-}(v)$ and $\deg_D^{+}(v)$, respectively, or simply by $\deg^{-}(v)$ and $\deg^{+}(v)$ when the digraph $D$ is clear from the context. Let   \(
N_D^{+}(v)=\{u\in V(D)\setminus\{v\}: (v,u)\in A(D)\} \text{ and }
N_D^{-}(v)=\{w\in V(D)\setminus\{v\}: (w,v)\in A(D)\}.
\)
The sets $N_D^{+}(v)$ and $N_D^{-}(v)$ are called the \textit{out-neighbourhood} and \textit{in-neighborhood} of $v$, respectively.
The vertices in the in-neighbourhood of a vertex are called its in-neighbours, and the vertices in its out-neighbourhood are called its out-neighbours.

The \emph{adjacency matrix} of $TT_n$ is $M(TT_n)=[m_{ij}]$, the $n\times n$
$0$--$1$ matrix defined by
\[
m_{ij}=
\begin{cases}
1, & \text{if } v_i \to v_j,\\
0, & \text{otherwise.}
\end{cases}
\]
Equivalently, $m_{ij}=1$ precisely when $i<j$ (and $m_{ii}=0$ for all $i$).

In a transitive tournament, every connected two-arc subdigraph on three distinct vertices \(v_i,v_j,v_k\) is isomorphic to exactly one of the following three digraphs: a \emph{chain} \((v_i\to v_j\to v_k)\), a \emph{fork} \((v_i\leftarrow v_j\to v_k)\), or a \emph{collider} \((v_i\to v_j\leftarrow v_k)\). Equivalently, chains, forks, and colliders are precisely the three possible orientations of the $2$-path \(P_3\). In this sense, they form the fundamental local structural building blocks for describing and analysing transitive tournaments. We sometimes call chains, colliders, and forks two-arc motifs or connected two-arc motifs. 

It is well known that the complete graph \(K_n\) can be decomposed into copies of \(P_3\) if and only if the number of edges of \(K_n\) is even; that is, if and only if \(n \equiv 0,1 \pmod{4}\).
We call \(n\) admissible if and only if \(n \equiv 0 \pmod{4}\) or \(n \equiv 1 \pmod{4}\).
This is because the number of two-arc motifs in \(TT_n\) is \(\frac{n(n-1)}{4}\), which must be an integer.
Since \(n\) and \(n-1\) have opposite parity, integrality forces one of them to be divisible by \(4\);
equivalently, \(n \equiv 0 \pmod{4}\) or \(n-1 \equiv 0 \pmod{4}\).

Let \(G\) be a connected graph. In 1980, Caro and Sch\"onheim proved that the obvious necessary condition \(|E(G)| \equiv 0 \pmod{2}\) is also sufficient for \(G\) to admit a decomposition into copies of the path \(P_3\) of length two. 

\begin{theorem}[\cite{CaroSchonheim1980}]
\label{T:2-path}
A connected graph \(G\) admits a \(P_3\)-decomposition if and only if
\[
|E(G)| \equiv 0 \pmod{2}.
\]
\end{theorem}

\begin{theorem} [\cite{Gorlichetal2007}]
\label{T: no single decomp}
For any admissible \(n\), the transitive tournament \(TT_n\) does not admit a decomposition into isomorphic copies of any one of the digraphs chains, colliders or forks.

\end{theorem}

\begin{proof}
Let $\{v_1,v_2,\ldots, v_{n-1},v_n\}$ be the topological ordering of the vertices of $TT_n$. Suppose that $TT_n$ has an $H$-decomposition where all the copies of $H$ are either colliders, chains, or forks.  Note that in the decomposition, the arc $v_{n-1} \rightarrow v_n$ can only be in a chain or a collider, such that $v_x \rightarrow v_{n-1} \rightarrow v_n$  or $v_{n-1} \rightarrow v_n \leftarrow v_y$ where $v_x,v_y \in \{v_1,v_2, \ldots, v_{n-2}\}$. 
Thus, it is a contradiction that a fork decomposition of $TT_n$ exists for any admissible $n$. 

Moreover, the arc $v_1 \rightarrow v_n$ can only be in a fork or a collider such that $v_1 \rightarrow v_n \leftarrow v_y$ or $v_x \leftarrow v_1 \rightarrow v_n$ where $v_x,v_y \in \{v_2, \ldots, v_{n-1}\}$. Therefore, a chain decomposition of $TT_n$ for any admissible order $n$ does not exist.

Furthermore, the arc $v_1 \rightarrow v_2$ can only be in a chain or a fork, such that $v_1 \rightarrow v_2 \rightarrow v_y$ or $v_x \leftarrow v_1 \rightarrow v_2$ where $v_x,v_y \in \{v_3, \dots, v_{n-1},v_n\}$. Therefore, it is a contradiction that a collider decomposition of $TT_n$ for any admissible order $n$ exists.
\end{proof}

By Theorem~\ref{T: no single decomp}, we know that for any admissible $n$,  $TT_n$ cannot be decomposed entirely into chains, entirely into colliders, or entirely into forks. However, mixed decompositions are possible; that is, \(TT_n\) may be decomposed into a combination of chains, colliders, and forks. Which is more like a packing of each kind. From now on, we will refer to such mixed decompositions as chain--collider--fork decompositions.

\begin{theorem}

For all admissible \(n\), the transitive tournament \(TT_n\) admits a chain--collider--fork decomposition.
\end{theorem}

\begin{proof}
Ignoring orientation, the transitive tournament \(TT_n\) is the complete graph \(K_n\). Since \(n\) is admissible, Theorem~\ref{T:2-path} implies that \(K_n\) admits a decomposition into paths of length two.

Now restore the orientation induced by \(TT_n\). Each copy of \(P_3\) then becomes a connected digraph with two arcs. Hence each such digraph is isomorphic to exactly one of a chain, a collider, or a fork. Therefore, the corresponding collection of subdigraphs forms a chain--collider--fork decomposition of \(TT_n\).
\end{proof}

\bigskip

\section{Dots-in-cells diagram}
\label{S:sec3}

In this section, we establish the key technical tool on which the remainder of the paper relies. We adapt the adjacency matrix of the transitive tournament $TT_n$ (with vertices $v_1,\dots,v_n$ in topological order) to construct a \emph{dots-in-cells} diagram as follows. Consider an $(n-1)\times(n-1)$ array whose rows are labelled $1,2,\dots,n-1$ and whose columns are labelled $2,3,\dots,n$.   The cell in row $i$ and column $j$ is identified with the ordered pair $(i,j)$, and thus represents the arc $v_i \to v_j$. We place a dot in this cell precisely when $v_i$ dominates $v_j$, that is, when the arc $v_i\to v_j$ is present in $TT_n$ (equivalently, when $i<j$). Thus, each dot represents the arc from the row-vertex to the column-vertex. Each dot represents a unique arc of $TT_n$. Therefore, there are $\frac{n(n-1)}{2}$ dots in total. In each row $i$ for $i \in \{1,2 \ldots n-1\}$ there are exactly $n-i$ dots (in columns $i+1,i+2,\dots,n$).

Using the dots-in-cells diagram, we interpret the three two-arc motifs as follows.
Pick two dots that share a common row (say the $i$th row). These correspond to arcs $v_i\to v_j$ and $v_i\to v_k$ for some $j,k>i$, and hence form a \emph{fork} with middle vertex (common tail) $v_i$, namely $v_j \leftarrow v_i \rightarrow v_k$.

Pick two dots that share a common column (say the $j$th column). These correspond
to arcs $v_i\to v_j$ and $v_k\to v_j$ for some $i,k<j$, and hence form a
\emph{collider} with middle vertex (common head) $v_j$, namely
$v_i \rightarrow v_j \leftarrow v_k$.

Finally, pick two dots diagonally so that one lies strictly below and to the right of the other such that the column index of the left dot equals the row index of the right dot. For an example, choose dots in cells $(i,j)$ and $(j,k)$ with $i<j<k$. These correspond to arcs $v_i\to v_j$ and $v_j\to v_k$, and hence form a \emph{chain} $v_i \rightarrow v_j \rightarrow v_k$.

\begin{example} Dots-in-cells diagram of $TT_8$

\begin{figure}[ht]
\begin{tikzpicture}[x=1cm,y=1cm]

\def\n{7}
\pgfmathtruncatemacro{\nminusone}{\n-1}

\draw (0,0) rectangle (\n,\n);

\foreach \k in {1,...,\nminusone} {
  \draw (\k,0) -- (\k,\n);
  \draw (0,\k) -- (\n,\k);
}

\foreach \j in {1,...,\n} {
  \pgfmathtruncatemacro{\lab}{\j+1}
  \node at (\j-0.5,\n+0.4) {\lab};
}

\foreach \i in {1,...,\n} {
  \node at (-0.4,\n-\i+0.5) {\i};
}

\foreach \i/\j in {
  1/1,1/2,1/3,1/4,1/5,1/6,1/7,
  2/2,2/3,2/4,2/5,2/6,2/7,
  3/3,3/4,3/5,3/6,3/7,
  4/4,4/5,4/6,4/7,
  5/5,5/6,5/7,
  6/6,6/7,
  7/7
} {
  \fill (\j-0.5,\n-\i+0.5) circle (2pt);
}


\draw[black, thick, rotate around={-5:(4,5.5)}]
  (4,5.5) ellipse [x radius=0.85, y radius=0.30];

\draw[black, thick, rotate around={2:(6.5,4)}]
  (6.5,4) ellipse [x radius=0.30, y radius=0.85];

\draw[black, thick, rotate around={-45:(5,2)}]
  (5,2) ellipse [x radius=0.85, y radius=0.30];

\end{tikzpicture}
\caption{Dots-in-cells diagram of $TT_8$ illustrating three two-arc motifs: a fork, $(v_5 \leftarrow v_2 \rightarrow v_6)$, shown by the horizontal selection; a collider, $(v_3 \rightarrow v_8 \leftarrow v_4)$, shown by the vertical selection; and a chain, $(v_5 \rightarrow v_6 \rightarrow v_7)$, shown by the diagonal selection.}
\end{figure}    
\end{example}

We now establish that, for every admissible $n$, the transitive tournament $TT_n$ admits a chain--collider--fork decomposition with the claimed numbers of chains, colliders, and forks. In particular, this provides a packing of $TT_n$ into chains, colliders, and forks.

\begin{lemma}\label{L: chain collider fork decomposition}

For every admissible $n$, the transitive tournament $TT_n$ admits a 
chain--collider--fork decomposition with $\left\lfloor \frac{n-1}{2}\right\rfloor$ chains, $\left\lfloor \frac{n}{4}\right\rfloor$ colliders, and $\frac{n(n-1)}{4}-\left\lfloor \frac{n-1}{2}\right\rfloor - \left\lfloor \frac{n}{4}\right\rfloor$ forks.
\end{lemma}

\begin{proof}
Let $\{v_1,v_2, \ldots, v_n\}$ be the topological ordering of $TT_n$ for any admissible $n$. Then $(v_i,v_j) \in A(TT_n)$ where $i <j$.  
We first construct the chains, then the forks, and finally the colliders. Consider the main diagonal of the dots-in-cell diagram of $TT_n$ . Each dot in the main diagonal in the cell $(i,i+1)$ for $i \in \{1, \ldots, n-1\}$, representing the arc $v_i \to v_{i+1}$. If we pick two consecutive dots in the main diagonal, they represent a chain of the form $v_i \rightarrow v_{i+1} \rightarrow v_{i+2}$. We will pick such consecutive pairs starting from the dot in $(1,2)$. Dots-in-cells diagram of $TT_n$ is an $n-1 \times n-1 $ array. Therefore, there are $n-1$ dots in the main diagonal. Since we choose pairs of distinct dots, there will be $\left\lfloor \frac{n-1}{2}\right\rfloor$ chains starting from $v_1 \rightarrow v_{2} \rightarrow v_{3}$.

If $n-1$ is odd, then the final diagonal dot, namely the dot in the cell $(n-1,n)$ representing the arc $v_{n-1}\to v_n$, is left unused. If $n-1$ is even, then every dot on the main diagonal is used to form a chain.
Hence, after the chains have been selected, the remaining dots in each row can be described as follows. If $n$ is odd, then row $i$ contains $n-i-1$ remaining dots for each $1\le i\le n-1$. If $n$ is even, then row $i$ contains $n-i-1$ remaining dots for each $1\le i\le n-2$, while the single dot in the cell $(n-1,n)$ remains in row $n-1$.

Next we construct the forks as follows. For each row $i$, pair its remaining dots consecutively from left to right:
\[
(i,i+2)\ \text{with}\ (i,i+3), \quad (i,i+4)\ \text{with}\ (i,i+5), \quad \dots
\]
Each such pair corresponds to a fork with tail vertex $v_i$.

If the number of remaining dots in row $i$ is even, then all of them are used to form forks. If the number of remaining dots in row $i$ is odd, then exactly one dot remains unpaired; since the pairing is performed from left to right, this unpaired dot is the rightmost remaining dot in that row, namely $(i,n)$, representing the arc $v_i \to v_n$.

Let $U$ denote the set of dots left unpaired after all possible forks have been formed. Then
\[
U=
\begin{cases}
\{(j,n): 1\le j\le n-1,\ j \text{ is even}\} \cup \{(n-1,n)\}, & \text{if } n \text{ is even},\\[1ex]
\{(j,n): 1\le j\le n-1,\ j \text{ is odd}\}, & \text{if } n \text{ is odd}.
\end{cases}
\]
Therefore,
\[
|U|=
\begin{cases}
\dfrac{n}{2}, & \text{if } n \text{ is even},\\[1ex]
\dfrac{n-1}{2}, & \text{if } n \text{ is odd}.
\end{cases}
\]
Since $n$ is admissible, we have $n\equiv 0 \pmod 4$ when $n$ is even and $n\equiv 1 \pmod 4$ when $n$ is odd. Hence
\[
|U|=2\left\lfloor \frac{n}{4}\right\rfloor.
\]

Now partition $U$ into $|U|/2=\left\lfloor \frac{n}{4}\right\rfloor$ disjoint pairs. Each pair consists of two dots in column $n$, say $(i,n)$ and $(k,n)$, corresponding to the arcs $v_i\to v_n$ and $v_k\to v_n$. Thus each such pair defines a collider
\[
v_i\to v_n \leftarrow v_k.
\]
Therefore,
\[
\text{number of colliders}=\left\lfloor \frac{n}{4}\right\rfloor.
\]

Finally, each chain, collider, and fork uses exactly two arcs, and $TT_n$ has $\frac{n(n-1)}{2}$ arcs. Hence the total number of two-arc motifs in the decomposition is
\[
\frac{n(n-1)}{4}.
\]
Therefore,
\[
\text{number of forks}
=
\frac{n(n-1)}{4}
-\left\lfloor \frac{n-1}{2}\right\rfloor
-\left\lfloor \frac{n}{4}\right\rfloor,
\]
as required.

\end{proof}

\section{Chains}
\label{S:sec4}

In this section, we focus on decompositions and packings in which the connected two-edge digraph is isomorphic to a chain, namely a directed path of length two. We investigate how many arc-disjoint copies of a chain can be packed into the transitive tournament \(TT_n\), and hence determine the chain-packing number of \(TT_n\). For a chain of the form $v_i \to v_j \to v_k$, we call $v_i \to v_j$ the first arc and $v_j \to v_k$ the second arc of the chain.

\begin{lemma}\label{L: more chains in collider chains}

For every admissible $n$, the transitive tournament $TT_n$ admits a 
chain--collider--fork decomposition with $\frac{1}{2}\left(\left\lceil \frac{n+1}{2}\right\rceil -1\right)$ colliders, $\frac{n(n-1)}{4}
- \frac{1}{2}\left(\left\lceil \frac{n+1}{2}\right\rceil -1\right)$ chains and $0$ forks.

\end{lemma}

\begin{proof}
Let $\{v_1,v_2, \ldots, v_n\}$ be the topological ordering of $TT_n$ for any admissible $n$. Then $(v_i,v_j) \in A(TT_n)$ where $i <j$. Let $v_i \rightarrow v_j \rightarrow v_k$ be a chain with $i <j < k$. Note that $v_j \in \{v_2,v_3, \ldots, v_{n-1}\}$. In other words, only $\{v_2,v_3, \ldots, v_{n-1}\}$ can be a centre vertex of a chain. Table~\ref{Tab:in degree out degree} lists the in-degrees and out-degrees of the vertices of $TT_n$ with respect to the topological ordering.

\begin{table}[h!]
    \centering
    \begin{tabular}{ccc}
        \hline
        Vertex & Out-degree  & In-degree \\
        \hline
        $v_n$ & $0$ & $n-1$ \\
        $v_{n-1}$ & $1$ & $n-2$ \\
        $v_{n-2}$  & $2$ & $n-3$ \\
        \vdots & \vdots & \vdots \\
        $v_{\left\lceil \frac{n+1}{2}\right\rceil}$ & $n-\left\lceil\frac{n+1}{2}\right\rceil$ & $\left\lceil\frac{n+1}{2}\right\rceil - 1$ \\
        \vdots & \vdots & \vdots \\
        $v_2$ & $n-2$ & $1$ \\
        $v_1$ & $n-1$ & $0$ \\
        
        \hline
    \end{tabular}
    \caption{In-degree and out-degree}
    \label{Tab:in degree out degree}
\end{table}

We now determine which vertices, under the topological ordering, have in-degree at least as large as their out-degree. For all $v_t \in V(TT_n)$, we have $\deg^{+}(v_t)=n-t$ and
\[
\deg^{-}(v_t)= (n-1)-\deg^{+}(v_t)=n-1-(n-t)=t-1,
 \quad \text{where } t \in [n].\]
Thus $\deg^{-}(v_t)\ge \deg^{+}(v_t)$ holds exactly when $t-1\ge n-t$, and hence when
\[
t\ge \left\lceil \frac{n+1}{2}\right\rceil .
\]

That is, for the vertices 
\[
v_n, v_{n-1},\, v_{n-2},\, \ldots,\, v_{\left\lceil \frac{n+1}{2}\right\rceil},
\]
the in-degree is greater than or equal to the out-degree. 
Conversely, for the vertices
\[
v_{1},\, v_{2},\, \ldots,\, v_{\left\lceil \frac{n+1}{2}\right\rceil-1},
\]
the in-degree is strictly less than the out-degree.

We first define the chains in the decomposition as follows. 
Let
\[
S_1=\{v_{n-1},v_{n-2},\ldots,v_{\lceil (n+1)/2\rceil}\}.
\]
For each $v_t\in S_1$, we have $\deg^{-}(v_t)\ge \deg^{+}(v_t)$, and hence there are sufficiently many in-neighbours available to pair with all out-neighbours of $v_t$. Accordingly, we choose $v_t$ as the centre of exactly $\deg^{+}(v_t)=n-t$ chains as follows.

Since
\[
N^+_{TT_n}(v_t)=\{v_{t+1},v_{t+2},\ldots,v_n\},
\]
the out-neighbours of $v_t$ are precisely
\[
v_t\to v_{t+1},\ v_t\to v_{t+2},\ \ldots,\ v_t\to v_n.
\]
These will serve as the second arcs of the chains. We pair them with $n-t$ distinct in-neighbours of $v_t$, which serve as the first arcs of the chains, chosen consecutively starting from $v_1$, namely
\[
v_1\to v_t,\ v_2\to v_t,\ \ldots,\ v_{n-t}\to v_t.
\]
This yields the chains
\[
v_1\to v_t\to v_{t+1},\quad
v_2\to v_t\to v_{t+2},\quad
\dots,\quad
v_{n-t}\to v_t\to v_n.
\]

Thus the chains in the decomposition of $TT_n$, where $v_t$ is the centre for $v_t \in S_1$, can be denoted as $v_i \to v_t \to v_{i+t}$ for $1\le i\le n-t$.

Now consider the set of vertices
\[
S_2=\{v_{\lceil (n+1)/2\rceil-1},\ldots,v_3,v_2\}.
\]
For each vertex $v_t\in S_2$, we have $\deg^{-}(v_t)<\deg^{+}(v_t)$, so $v_t$ has fewer in-neighbours than out-neighbours. We therefore make $v_t$ the centre of exactly $\deg^{-}(v_t)=t-1$ chains, so that every in-neighbour of $v_t$ is used as the first arc of some chain.

More precisely, for each $v_t\in S_2$,
\[
N^-_{TT_n}(v_t)=\{v_1,v_2,\dots,v_{t-1}\},
\]
so $|N^-_{TT_n}(v_t)|=t-1$. We pair these $t-1$ in-neighbours with $t-1$ distinct out-neighbours of $v_t$, chosen consecutively in descending order starting from $v_{n-1}$. Thus, we obtain the chains
\[
v_1\to v_t\to v_{n-1},\quad
v_2\to v_t\to v_{n-2},\quad
\dots,\quad
v_{t-1}\to v_t\to v_{n-(t-1)}.
\]

Therefore, the chains in the decomposition of $TT_n$, where $v_t$ is the centre for $v_t \in S_2$, can be denoted as $v_i \to v_t \to v_{n-i}$ for $1\le i\le t-1$.

We claim that the constructed chains are pairwise arc-disjoint. Indeed, consider any arc $v_i\to v_j$ with $1\le i<j\le n-1$.
If $v_j\in S_2$, then $i\le j-1$, and by construction all in-neighbours of $v_j$ are used exactly once as the first arcs of chains centred at $v_j$. Hence, $v_i\to v_j$ appears exactly once as the first arc.

Now suppose that $v_j\in S_1$. If $i\le n-j$, then by the construction for $S_1$, the arc $v_i\to v_j$ is used exactly once as the first arc of a chain centred at $v_j$. On the other hand, if $i>n-j$, then $j\ge n-i+1$, so the arc $v_i\to v_j$ is one of the out-neighbours selected in the construction for the centre $v_i$; hence it is used exactly once as a second arc of a chain centred at $v_i$.

Therefore, every arc $v_i\to v_j$ with $j<n$ belongs to exactly one constructed chain. It follows that the chains are pairwise arc-disjoint. The only arcs not used in any chain are those of the form $v_i\to v_n$ with $v_i\in S_2\cup\{v_1\}$.

Hence, the number of unused arcs is
\[
|S_2|+1=\left\lceil \frac{n+1}{2}\right\rceil -1.
\]
Since $n$ is admissible, we have $n\equiv 0 \text{ or } 1 \pmod{4}$. 
In both cases, the quantity 
\[
\left\lceil \frac{n+1}{2}\right\rceil -1
\]
is even. Therefore, these remaining arcs can be partitioned into pairs.

Each unused arc has the form $v_i \to v_n$, so any pair of such arcs 
$
v_i \to v_n 
\quad \text{and} \quad 
v_j \to v_n
$
shares the common head $v_n$. Hence each pair forms a collider $v_i \to v_n \leftarrow v_j$ with centre $v_n$.

Consequently, we obtain
\[
\frac{1}{2}\left(\left\lceil \frac{n+1}{2}\right\rceil -1\right)
\]
colliders, completing the construction. Therefore, the number of chains in the decomposition is
\[
\frac{n(n-1)}{4}
-
\frac{1}{2}\left(\left\lceil \frac{n+1}{2}\right\rceil -1\right).
\]

\end{proof}

\begin{example}
Table~\ref{tab:Chain--collider decomposition of $TT_8$} illustrates a chain--collider--fork decomposition of $TT_8$ with $12$ chains, $2$ colliders and $0$ forks. 

\begin{table}[h!]
    \centering
    \begin{tabular}{|c c|c|}
        \hline
        \multicolumn{2}{|c|}{Chains} & Colliders \\ \hline
        $v_1 \to v_7 \to v_8$ & $v_1 \to v_4 \to v_7 $ & $v_1 \to v_8 \leftarrow v_2$ \\ 
        $v_1 \to v_6 \to v_7$ & $v_2 \to v_4 \to v_6$ & $v_3 \to v_8 \leftarrow v_4$ \\ 
        $v_2 \to v_6 \to v_8$ & $v_3 \to v_4 \to v_5$ &  \\ 
        $v_1 \to v_5 \to v_6$ & $v_1 \to v_3 \to v_7$ &  \\ 
        $v_2 \to v_5 \to v_7$ & $v_2 \to v_3 \to v_6$ &  \\ 
        $v_3 \to v_5 \to v_8$ & $v_1 \to v_2 \to v_7$ &  \\ \hline
    \end{tabular}
    \caption{Chain--collider--fork decomposition of $TT_8$}
    \label{tab:Chain--collider decomposition of $TT_8$}
\end{table}

\end{example}

\begin{example}
Table~\ref{tab:Chain--collider decomposition of $TT_9$} illustrates a chain--collider--fork decomposition of $TT_9$ with $16$ chains, $2$ colliders and $0$ forks. 

\begin{table}[h!]
    \centering

    \begin{tabular}{|c c|c|}
        \hline
        \multicolumn{2}{|c|}{Chains} & Colliders \\ \hline
        $v_1 \to v_8 \to v_9$ & $v_1 \to v_4 \to v_8 $ & $v_1 \to v_9 \leftarrow v_2$ \\ 
        $v_1 \to v_7 \to v_8$ & $v_2 \to v_4 \to v_7$ & $v_3 \to v_9 \leftarrow v_4$ \\ 
        $v_2 \to v_7 \to v_9$ & $v_3 \to v_4 \to v_6$ &  \\ 
        $v_1 \to v_6 \to v_7$ & $v_1 \to v_3 \to v_8$ &  \\ 
        $v_2 \to v_6 \to v_8$ & $v_2 \to v_3 \to v_7$ &  \\ 
        $v_3 \to v_6 \to v_9$ & $v_1 \to v_2 \to v_8$ &  \\ 
        $v_1 \to v_5 \to v_6$ &  &  \\ 
        $v_2 \to v_5 \to v_7$ &  &  \\ 
        $v_3 \to v_5 \to v_8$ & &  \\ 
        $v_4 \to v_5 \to v_9$ &  &  \\ 
        \hline
    \end{tabular}
    \caption{Chain--collider--fork decomposition of $TT_9$}
    \label{tab:Chain--collider decomposition of $TT_9$}
\end{table}

\end{example}

\begin{theorem}
For every admissible $n$, the maximum number of chains in a chain-collider-fork decomposition of $TT_n$ is $\frac{n}{4}(n-2)$ when $n$ is even, and $\frac{(n-1)^2}{4}$ when $n$ is odd. 
\end{theorem}

\begin{proof}
By Lemma~\ref{L: more chains in collider chains}, there exists a chain--collider--fork decomposition of $TT_n$ with
\[
\frac{n(n-1)}{4}
-
\frac{1}{2}\left(\left\lceil \frac{n+1}{2}\right\rceil -1\right) \quad \text{chains.}
\]
Note that, \[
\frac{n(n-1)}{4}
-
\frac{1}{2}\left(\left\lceil \frac{n+1}{2}\right\rceil -1\right)
=
\begin{cases}
\frac{n}{4}(n-2), & \text{when $n$ is even},\\[1ex]
\frac{(n-1)^2}{4}, & \text{when $n$ is odd}.
\end{cases}
\]

It is sufficient to show that no such decomposition can contain more chains.
Let $\{v_1,v_2,\ldots,v_n\}$ be the topological ordering of $TT_n$, and define
\[
S_1=\{\,v_{n-1},\,v_{n-2},\,\ldots,\,v_{\lceil (n+1)/2\rceil}\,\},
\qquad
S_2=\{\,v_{\lceil (n+1)/2\rceil-1},\,\ldots,\,v_3,\,v_2\,\}.
\]
Note that neither $v_1$ nor $v_n$ can be the centre of a chain, since $\deg^{-}(v_1)=0$ and $\deg^{+}(v_n)=0$.
Hence every centre of a chain must belong to $S_1\cup S_2=\{v_2,v_3,\ldots,v_{n-1}\}$.

Now let $v_t\in S_1\cup S_2$. If $v_t\in S_1$, then $\deg^{-}(v_t)\geq \deg^{+}(v_t)$, so the number of chains centred at $v_t$ is at most $\deg^{+}(v_t)=n-t$, since each chain centred at $v_t$ uses a distinct out-neighbour of $v_t$. If $v_t\in S_2$, then $\deg^{-}(v_t)<\deg^{+}(v_t)$, so the number of chains centred at $v_t$ is at most $\deg^{-}(v_t)=t-1$, since each chain centred at $v_t$ uses a distinct in-neighbour of $v_t$.

Therefore, for every $v_t\in \{v_2,\ldots,v_{n-1}\}$, the number of chains with centre $v_t$ is at most
\[
\min\{\deg^{-}(v_t),\deg^{+}(v_t)\}
=
\begin{cases}
n-t, & \text{if } v_t\in S_1,\\
t-1, & \text{if } v_t\in S_2.
\end{cases}
\]
In Lemma~\ref{L: more chains in collider chains}, the construction attains exactly this bound at every possible centre vertex:
each $v_t\in S_1$ is the centre of exactly $n-t=\deg^{+}(v_t)$ chains, and each $v_t\in S_2$ is the centre of exactly $t-1=\deg^{-}(v_t)$ chains.

Hence no vertex can be the centre of more chains than in that construction, and so no chain--collider--fork decomposition of $TT_n$ can contain more chains overall. Thus the construction in Lemma~\ref{L: more chains in collider chains} is maximum.
\end{proof}

\begin{note}
The construction above shows that the chains are chosen optimally by considering the in-degrees and out-degrees of the vertices that can serve as the centre of a chain. Therefore, for any $n$, the chain-packing number of $TT_n$ is $n(n-2)/4$ when $n$ is even; $(n-1)^2/4$ when $n$ is odd.
\end{note}

\section{Colliders}
\label{S:sec5}

In this section, we focus on decompositions and packings in which the connected two-edge digraph is isomorphic to a collider. We investigate how many arc-disjoint copies of a collider can be packed into the transitive tournament \(TT_n\), and hence determine the collider-packing number of \(TT_n\). The constructions in this section are based on the dots-in-cells diagram of $TT_n$ mentioned in Section~\ref{S:sec3}.

\begin{lemma}\label{L: more collides in collider fork}
For every admissible $n$, the transitive tournament $TT_n$ admits a chain--collider--fork decomposition with $\left\lfloor \frac{n}{4}\right\rfloor$ forks, $\frac{n(n-1)}{4}-\left\lfloor \frac{n}{4}\right\rfloor$ colliders and $0$ chains.

\end{lemma}

\begin{proof}
Consider the dots-in-cells diagram of $TT_n$, where each dot represents a unique arc of $TT_n$.
Thus there are $\frac{n(n-1)}{2}$ dots in total, and each column $i\in\{2,3,\dots,n\}$ contains exactly $i-1$ dots,
corresponding to the arcs $(1,i),(2,i),\dots,(i-1,i)$.

For each column $i$, pair its dots consecutively from bottom to top:
\[
(i-1,i)\ \text{with}\ (i-2,i),\quad (i-3,i)\ \text{with}\ (i-4,i),\ \dots
\]
Each such pair corresponds to a collider with centre vertex $i$.
If $i$ is odd then $i-1$ is even, so column $i$ is fully paired and contributes only colliders.
If $i$ is even then $i-1$ is odd, so exactly one dot remains unpaired in column $i$; since we paired from bottom to top,
this unpaired dot is the topmost one, namely $(1,i)$.
Hence the set of unpaired dots is precisely
\[
U=\{(1,i):\ i\ \text{even}\}.
\]
There are $\lfloor n/2\rfloor$ dots in $U$. Since $n$ is admissible, that is $n \equiv 0,1 \mod{4}$, $\lfloor n/2\rfloor$ is even, so we can partition $U$ into
$\lfloor n/2\rfloor/2$ pairs. Each such pair corresponds to a fork with middle vertex $v_1$.
Therefore,
\[
\text{number of forks}=\frac{1}{2}\left\lfloor \frac{n}{2}\right\rfloor=\left\lfloor \frac{n}{4}\right\rfloor.
\]

Thus, every arc of \(TT_n\) has been used exactly once in the construction, and each arc belongs to precisely one motif, either a fork or a collider. Since each such motif contains exactly two arcs, while \(TT_n\) has $\frac{n(n-1)}{2}$ arcs,

\[
\text{number of colliders}=\frac{n(n-1)}{4}-\left\lfloor \frac{n}{4}\right\rfloor,
\]
as required.
Further, this simplified to $\frac{n}{4}(n-2)$ when $n$ is even, and $\frac{(n-1)^2}{4}$ when $n$ is odd.
\end{proof}

\begin{example}
Table~\ref{tab:collider--fork decomposition of $TT_8$} illustrates a chain--collider--fork decomposition of $TT_8$ with $12$ colliders, $2$ forks and $0$ chains. 

\begin{figure}[ht]
\begin{tikzpicture}[x=1cm,y=1cm]

\def\n{7}
\pgfmathtruncatemacro{\nminusone}{\n-1}

\draw (0,0) rectangle (\n,\n);

\foreach \k in {1,...,\nminusone} {
  \draw (\k,0) -- (\k,\n);
  \draw (0,\k) -- (\n,\k);
}

\foreach \j in {1,...,\n} {
  \pgfmathtruncatemacro{\lab}{\j+1}
  \node at (\j-0.5,\n+0.4) {\lab};
}

\foreach \i in {1,...,\n} {
  \node at (-0.4,\n-\i+0.5) {\i};
}

\foreach \i/\j in {
  1/1,1/2,1/3,1/4,1/5,1/6,1/7,
  2/2,2/3,2/4,2/5,2/6,2/7,
  3/3,3/4,3/5,3/6,3/7,
  4/4,4/5,4/6,4/7,
  5/5,5/6,5/7,
  6/6,6/7,
  7/7
} {
  \fill (\j-0.5,\n-\i+0.5) circle (2pt);
}

\foreach \i/\j/\ang in {
  1/2/5,
  2/3/-8,
  1/4/-2,
  2/5/6,
  1/6/2,
  2/7/-5,
  3/4/0,
  3/6/3,
  4/5/-7,
  4/7/0,
  5/6/4,
  6/7/-4
} {
  \draw[thick, rotate around={\ang:(\j-0.5,\n-\i)}]
    (\j-0.5,\n-\i) ellipse [x radius=0.33, y radius=0.80];
}

\end{tikzpicture}
\caption{A dots-in-cells diagram for \(TT_8\), illustrating the construction of colliders.}
\end{figure}

\begin{table}[ht]
    \centering
    
    \label{tab:placeholder_label}
    \begin{tabular}{|c c|c|}
        \hline
        \multicolumn{2}{|c|}{Colliders} & Forks \\ \hline
        $v_1 \to v_3 \leftarrow v_2$ & $v_1 \to v_7 \leftarrow v_2$ & $v_2 \leftarrow v_1 \to v_4$ \\ 
        $v_2 \to v_4 \leftarrow v_3$ & $v_3 \to v_7 \leftarrow v_4$ & $v_6 \leftarrow v_1 \to v_8$ \\ 
        $v_1 \to v_5 \leftarrow v_2$ & $v_5 \to v_7 \leftarrow v_6$ &  \\ 
        $v_3 \to v_5 \leftarrow v_4$ & $v_2 \to v_8 \leftarrow v_3$ &  \\ 
        $v_2 \to v_6 \leftarrow v_3$ & $v_4 \to v_8 \leftarrow v_5$ &  \\ 
        $v_4 \to v_6 \leftarrow v_5$ & $v_6 \to v_8 \leftarrow v_7$ &  \\ \hline
    \end{tabular}
    \caption{Chain--collider--fork decomposition of $TT_8$}
    \label{tab:collider--fork decomposition of $TT_8$}
\end{table}

\end{example}

\newpage

\begin{example}
Table~\ref{tab:collider--fork decomposition of $TT_9$} illustrates a chain--collider--fork decomposition of $TT_9$ with $16$ colliders, $2$ forks and $0$ chains. 

\begin{figure}[ht]
\begin{tikzpicture}[x=1cm,y=1cm]

\def\ncols{8}
\def\nrows{8}

\draw (0,0) rectangle (\ncols,\nrows);

\foreach \k in {1,...,7} {
  \draw (\k,0) -- (\k,\nrows);
}

\foreach \k in {1,...,7} {
  \draw (0,\k) -- (\ncols,\k);
}

\foreach \j in {1,...,8} {
  \pgfmathtruncatemacro{\lab}{\j+1}
  \node at (\j-0.5,\nrows+0.4) {\lab};
}

\foreach \i in {1,...,8} {
  \node at (-0.4,\nrows-\i+0.5) {\i};
}

\foreach \i/\j in {
  1/1,1/2,1/3,1/4,1/5,1/6,1/7,1/8,
  2/2,2/3,2/4,2/5,2/6,2/7,2/8,
  3/3,3/4,3/5,3/6,3/7,3/8,
  4/4,4/5,4/6,4/7,4/8,
  5/5,5/6,5/7,5/8,
  6/6,6/7,6/8,
  7/7,7/8,
  8/8
} {
  \fill (\j-0.5,\nrows-\i+0.5) circle (2pt);
}

\foreach \i/\j/\ang in {
  1/2/5,
  2/3/-8,
  1/4/-2,
  3/4/0,
  2/5/6,
  4/5/0,
  1/6/2,
  3/6/0,
  5/6/-3,
  2/7/-5,
  4/7/0,
  6/7/-2,
  1/8/2,
  3/8/0,
  5/8/3,
  7/8/-3
} {
  \draw[thick, rotate around={\ang:(\j-0.5,\nrows-\i)}]
    (\j-0.5,\nrows-\i) ellipse [x radius=0.33, y radius=0.80];
}

\end{tikzpicture}
\caption{A dots-in-cells diagram for \(TT_9\), illustrating the construction of colliders.}
\end{figure}
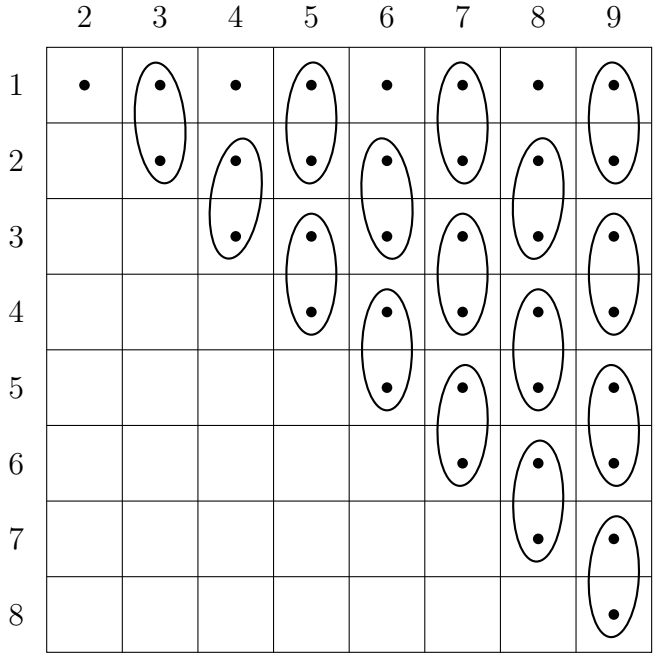

\begin{table}[h!]
    \centering
    \begin{tabular}{|c c|c|}
        \hline
        \multicolumn{2}{|c|}{Colliders} & Forks \\ \hline
        $v_1 \to v_3 \leftarrow v_2$ & $v_5 \to v_7 \leftarrow v_6$ & $v_2 \leftarrow v_1 \to v_4$ \\ 
        $v_2 \to v_4 \leftarrow v_3$ & $v_2 \to v_8 \leftarrow v_3$ & $v_6 \leftarrow v_1 \to v_8$ \\ 
        $v_1 \to v_5 \leftarrow v_2$ & $v_4 \to v_8 \leftarrow v_5$ & \\ 
        $v_3 \to v_5 \leftarrow v_4$ & $v_6 \to v_8 \leftarrow v_7$ & \\ 
        $v_2 \to v_6 \leftarrow v_3$ & $v_1 \to v_9 \leftarrow v_2$ & \\ 
        $v_4 \to v_6 \leftarrow v_5$ & $v_3 \to v_9 \leftarrow v_4$ & \\ 
        $v_1 \to v_7 \leftarrow v_2$ & $v_5 \to v_9 \leftarrow v_6$ & \\ 
        $v_3 \to v_7 \leftarrow v_4$ & $v_7 \to v_9 \leftarrow v_8$ & \\ \hline
    \end{tabular}
    \caption{Chain--collider--fork decomposition of $TT_9$}
    \label{tab:collider--fork decomposition of $TT_9$}
\end{table}
    
\end{example}

\begin{theorem}
For every admissible $n$, the maximum number of colliders in a chain-collider-fork decomposition of $TT_n$ is $\frac{n}{4}(n-2)$ when $n$ is even, and $\frac{(n-1)^2}{4}$ when $n$ is odd.
\end{theorem}

\begin{proof}
Let $V(TT_n)=\{v_1,v_2,\ldots,v_n\}$ be the topological ordering of $TT_n$. With this ordering, a collider has the form $v_j \rightarrow v_i \leftarrow v_k$, for $i > j,k$. Therefore, $\deg^-(v_i) \geq 2$.
Hence, only the vertices $v_3,\ldots,v_{n-2}, v_{n-1},v_n$ can occur as the middle vertex of a collider.
For each $v_i \in \{v_3,\ldots,v_{n-2},v_{n-1},v_n\}$, $\deg^-(v_i) = i-1$. Therefore, the number of colliders with the middle vertex $v_i$ is at most $\left\lfloor \frac{i-1}{2}\right\rfloor$.
Thus, the maximum possible number of colliders is:
\[
\text{total colliders}= \sum_{i=3}^{n}\left\lfloor \frac{i-1}{2}\right\rfloor =
\begin{cases}
\displaystyle
\sum_{i=3}^{n}\frac{i-1}{2}\;-\;\frac{1}{2}\cdot\frac{n-2}{2}
\;=\;\frac{n(n-2)}{4},
& \text{if $n$ is even},\\[2ex]
\displaystyle
\sum_{i=3}^{n}\frac{i-1}{2}\;-\;\frac{1}{2}\cdot\frac{n-3}{2}
\;=\;\frac{(n-1)^2}{4},
& \text{if $n$ is odd}.
\end{cases}
\]

In Lemma~\ref{L: more collides in collider fork}, we constructed  a chain--collider-fork decomposition of $TT_n$ with $\frac{n}{4}(n-2)$ colliders when $n$ is even and $\frac{(n-1)^2}{4}$ when $n$ is odd. Thus, the above upper bounds are tight, proving the claim.     \end{proof}

\begin{note}
The construction above shows that the colliders are chosen optimally by selecting all possible pairs of dots in each column of the dots-in-cells diagram. Therefore, for any $n$, the collider-packing number of $TT_n$ is $n(n-2)/4$ when $n$ is even; $(n-1)^2/4$ when $n$ is odd.
\end{note}

\section{Forks}
\label{S:sec6}

In this section, we focus on decompositions and packings in which the connected two-edge digraph is isomorphic to a fork. We investigate how many arc-disjoint copies of a fork can be packed into the transitive tournament \(TT_n\), and hence determine the fork-packing number of \(TT_n\). The constructions in this section are based on the dots-in-cells diagram of $TT_n$ mentioned in Section~\ref{S:sec3}.

\begin{lemma}\label{L: more forks in collider fork}
For every admissible $n$, the transitive tournament $TT_n$ admits a chain--collider--fork decomposition with $\left\lfloor \frac{n}{4}\right\rfloor$ colliders, $\frac{n(n-1)}{4}-\left\lfloor \frac{n}{4}\right\rfloor$ forks and $0$ chains.

\end{lemma}
\begin{proof}
Consider the dots-in-cells diagram of $TT_n$, where each dot represents a unique arc of $TT_n$.
Thus there are $\frac{n(n-1)}{2}$ dots in total, and each row $i\in\{1,2,3,\dots,n-1\}$ contains exactly $n-i$ dots,
corresponding to the arcs $(i,i+1),(i,i+2),\dots,(i,n)$.
For each row $i$, pair its dots consecutively from left to right:
\[
(i,i+1)\ \text{with}\ (i,i+2),\quad (i,i+3)\ \text{with}\ (i,i+4),\ \dots
\]
Each such pair corresponds to a fork with middle (tail) vertex $v_i$.
If $n-i$ is even then row $i$ is fully paired and contributes only forks.
If $n-i$ is odd then exactly one dot remains unpaired in row $i$; since we paired from left to right,
this unpaired dot is the rightmost one, namely $(i,n)$.
Hence the set of unpaired dots is precisely
\[
U=\{(i,n):\ n-i\ \text{odd}\}.
\]
Equivalently, $U$ consists of the dots in column $n$ coming from rows $i$ of parity opposite to $n$.
Therefore,
\[
|U|=
\begin{cases}
\frac{n}{2}, & \text{if $n$ is even},\\[0.5ex]
\frac{n-1}{2}, & \text{if $n$ is odd},
\end{cases}
\qquad\text{so}\qquad
|U|=2\left\lfloor\frac{n}{4}\right\rfloor,
\]
where the last equality uses admissibility (so $n\equiv 0\pmod4$ if $n$ is even, and $n\equiv 1\pmod4$ if $n$ is odd).
Partition $U$ into $|U|/2=\lfloor n/4\rfloor$ disjoint pairs.
Each pair consists of two dots in the column $n$, corresponding to arcs $v_i\to v_n$ and $v_k\to v_n$ with common head $v_n$,
and hence defines a collider $v_i\to v_n \leftarrow v_k$.
Thus,
\[
\text{number of colliders}=\left\lfloor \frac{n}{4}\right\rfloor.
\]

Thus, every arc of \(TT_n\) has been used exactly once in the construction, and each arc belongs to precisely one motif, either a fork or a collider. Since each such motif contains exactly two arcs, while \(TT_n\) has $\frac{n(n-1)}{2}$ arcs,

\[
\text{number of forks}=\frac{n(n-1)}{4}-\left\lfloor \frac{n}{4}\right\rfloor,
\]
as required.
Further, this simplified to $\frac{n}{4}(n-2)$ when $n$ is even, and $\frac{(n-1)^2}{4}$ when $n$ is odd.
\end{proof}

\begin{example}
Table~\ref{tab:fork decomposition of $TT_8$} illustrates a chain--collider--fork decomposition of $TT_8$ with $12$ forks, $2$ colliders and $0$ chains.

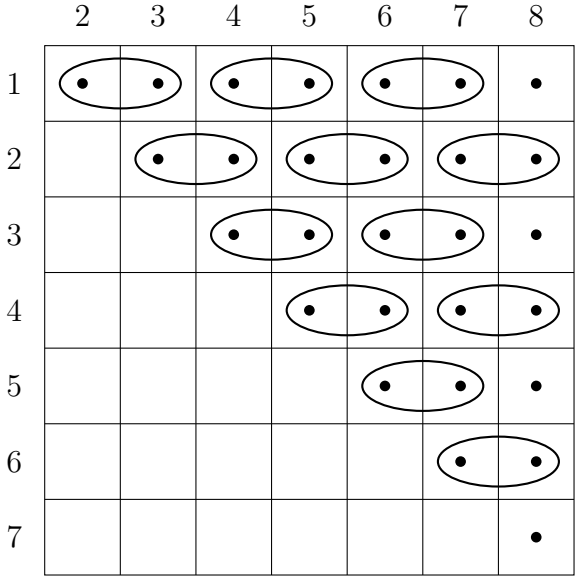
\begin{figure}[ht]

\begin{tikzpicture}[x=1cm,y=1cm]

\def\n{7}
\pgfmathtruncatemacro{\nminusone}{\n-1}

\draw (0,0) rectangle (\n,\n);

\foreach \k in {1,...,\nminusone} {
  \draw (\k,0) -- (\k,\n);
  \draw (0,\k) -- (\n,\k);
}

\foreach \j in {1,...,\n} {
  \pgfmathtruncatemacro{\lab}{\j+1}
  \node at (\j-0.5,\n+0.4) {\lab};
}

\foreach \i in {1,...,\n} {
  \node at (-0.4,\n-\i+0.5) {\i};
}

\foreach \i/\j in {
  1/1,1/2,1/3,1/4,1/5,1/6,1/7,
  2/2,2/3,2/4,2/5,2/6,2/7,
  3/3,3/4,3/5,3/6,3/7,
  4/4,4/5,4/6,4/7,
  5/5,5/6,5/7,
  6/6,6/7,
  7/7
} {
  \fill (\j-0.5,\n-\i+0.5) circle (2pt);
}

\foreach \i/\j/\ang in {
  1/1/0,
  1/3/0,
  1/5/0,
  2/2/0,
  2/4/0,
  2/6/0,
  3/3/0,
  3/5/0,
  4/4/0,
  4/6/0,
  5/5/0,
  6/6/0
} {
  \draw[thick, rotate around={\ang:(\j,\n-\i+0.5)}]
    (\j,\n-\i+0.5) ellipse [x radius=0.80, y radius=0.33];
}

\end{tikzpicture}
\caption{A dots-in-cells diagram for \(TT_8\), illustrating the construction of forks.}
\end{figure}

\begin{table}[h!]
    \centering
    
    \label{tab:placeholder_label}
    \begin{tabular}{|c c|c|}
        \hline
        \multicolumn{2}{|c|}{Forks} & Colliders \\ \hline
        $v_2 \leftarrow v_1 \rightarrow v_3$ & $v_4 \leftarrow v_1 \rightarrow v_5$ & $v_1 \to v_8 \leftarrow v_3$ \\ 
        $v_6 \leftarrow v_1 \rightarrow v_7$ & $v_3 \leftarrow v_2 \rightarrow v_4$ & $v_5 \to v_8 \leftarrow v_7$ \\ 
        $v_5 \leftarrow v_2 \rightarrow v_6$ & $v_7 \leftarrow v_2 \rightarrow v_8$ &  \\ 
        $v_4 \leftarrow v_3 \rightarrow v_5$ & $v_6 \leftarrow v_3 \rightarrow v_7$ &  \\ 
        $v_5 \leftarrow v_4 \rightarrow v_6$ & $v_7 \leftarrow v_4 \rightarrow v_8$ &  \\ 
        $v_6 \leftarrow v_5 \rightarrow v_7$ & $v_7 \leftarrow v_6 \rightarrow v_8$ &  \\ \hline
    \end{tabular}
     \caption{Chain--collider--fork decomposition of $TT_8$}
     \label{tab:fork decomposition of $TT_8$}
\end{table}

\end{example}

\newpage

\begin{example}
Table~\ref{tab:fork decomposition of $TT_9$} illustrates a chain--collider--fork decomposition of $TT_9$ with $16$ forks, $2$ colliders and $0$ chains. 

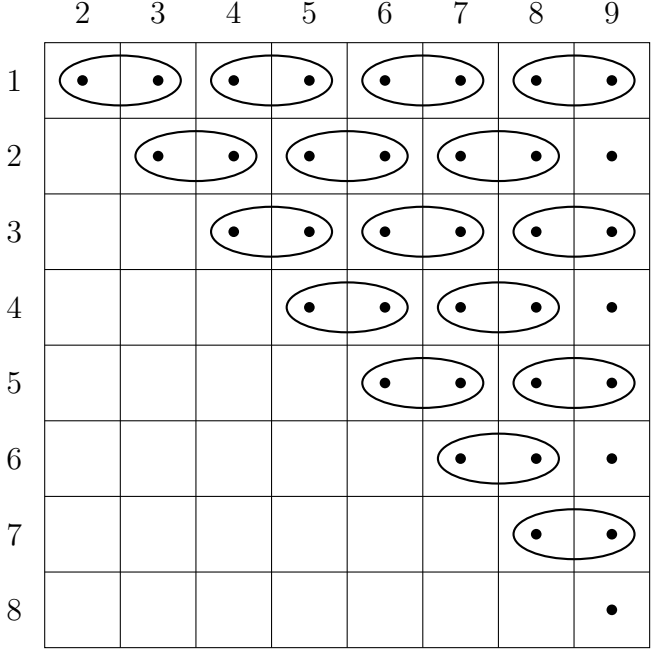
\begin{figure}[ht]
\begin{tikzpicture}[x=1cm,y=1cm]

\def\n{8}
\pgfmathtruncatemacro{\nminusone}{\n-1}

\draw (0,0) rectangle (\n,\n);

\foreach \k in {1,...,\nminusone} {
  \draw (\k,0) -- (\k,\n);
  \draw (0,\k) -- (\n,\k);
}

\foreach \j in {1,...,\n} {
  \pgfmathtruncatemacro{\lab}{\j+1}
  \node at (\j-0.5,\n+0.4) {\lab};
}

\foreach \i in {1,...,\n} {
  \node at (-0.4,\n-\i+0.5) {\i};
}

\foreach \i/\j in {
  1/1,1/2,1/3,1/4,1/5,1/6,1/7,1/8,
  2/2,2/3,2/4,2/5,2/6,2/7,2/8,
  3/3,3/4,3/5,3/6,3/7,3/8,
  4/4,4/5,4/6,4/7,4/8,
  5/5,5/6,5/7,5/8,
  6/6,6/7,6/8,
  7/7,7/8,
  8/8
} {
  \fill (\j-0.5,\n-\i+0.5) circle (2pt);
}

\foreach \i/\j/\ang in {
  1/1/0,
  1/3/0,
  1/5/0,
  1/7/0,
  2/2/0,
  2/4/0,
  2/6/0,
  3/3/0,
  3/5/0,
  3/7/0,
  4/4/0,
  4/6/0,
  5/5/0,
  5/7/0,
  6/6/0,
  7/7/0
} {
  \draw[thick, rotate around={\ang:(\j,\n-\i+0.5)}]
    (\j,\n-\i+0.5) ellipse [x radius=0.80, y radius=0.33];
}

\end{tikzpicture}
\caption{A dots-in-cells diagram for \(TT_9\), illustrating the construction of forks.}
\end{figure}

\begin{table}[h!]
    \centering
    \begin{tabular}{|c c|c|}
        \hline
        \multicolumn{2}{|c|}{Forks} & Colliders \\ \hline
        $v_2 \leftarrow v_1 \rightarrow v_3$ & $v_5 \leftarrow v_3 \rightarrow v_6$ & $v_2 \to v_9 \leftarrow v_4$ \\ 
        $v_4 \leftarrow v_1 \rightarrow v_5$ & $v_7 \leftarrow v_3 \rightarrow v_8$ & $v_6 \to v_9 \leftarrow v_8$ \\ 
        $v_6 \leftarrow v_1 \rightarrow v_7$ & $v_6 \leftarrow v_4 \rightarrow v_7$ &  \\ 
        $v_8 \leftarrow v_1 \rightarrow v_9$ & $v_8 \leftarrow v_4 \rightarrow v_9$ &  \\ 
        $v_3 \leftarrow v_2 \rightarrow v_4$ & $v_7 \leftarrow v_5 \rightarrow v_8$ &  \\ 
        $v_5 \leftarrow v_2 \rightarrow v_6$ & $v_8 \leftarrow v_5 \rightarrow v_9$ &  \\ 
        $v_7 \leftarrow v_2 \rightarrow v_8$ & $v_8 \leftarrow v_6 \rightarrow v_7$ &  \\ 
        $v_4 \leftarrow v_3 \rightarrow v_5$ & $v_8 \leftarrow v_7 \rightarrow v_9$ &  \\ \hline
    \end{tabular}
     \caption{Chain--collider--fork decomposition of $TT_9$}
     \label{tab:fork decomposition of $TT_9$}
\end{table}

\end{example}

\begin{theorem}
For every admissible $n$, the maximum number of forks in a chain-collider-fork decomposition of $TT_n$ is $\frac{n}{4}(n-2)$ when $n$ is even, and $\frac{(n-1)^2}{4}$ when $n$ is odd.

\end{theorem}


\begin{proof}

Let $V(TT_n)=\{v_1,v_2,\ldots,v_n\}$ be the topological ordering of $TT_n$. With this ordering, a fork has the form $v_j \leftarrow v_i \to v_k$, for $i < j,k$. Therefore, $\deg^+(v_i) \geq 2$.
Hence, only the vertices $v_1,v_2,\ldots,v_{n-2}$ can occur as the middle vertex of a fork.
For each $v_i \in \{v_1,\ldots,v_{n-2}\}$, $\deg^+(v_i) = n-i$.
Therefore, the number of forks with middle vertex $v_i$
is at most $\left\lfloor \frac{n-i}{2}\right\rfloor$.
Thus, the maximum possible number of forks is:
\[
\text{total forks}= \sum_{i=1}^{n-2}\left\lfloor \frac{n-i}{2}\right\rfloor =
\begin{cases}
\displaystyle
\sum_{i=1}^{n-2}\frac{n-i}{2}\;-\;\frac{1}{2}\cdot\frac{n-2}{2}
\;=\;\frac{n(n-2)}{4},
& \text{if $n$ is even},\\[2ex]
\displaystyle
\sum_{i=1}^{n-2}\frac{n-i}{2}\;-\;\frac{1}{2}\cdot\frac{n-3}{2}
\;=\;\frac{(n-1)^2}{4},
& \text{if $n$ is odd}.
\end{cases}
\]

In Lemma~\ref{L: more forks in collider fork}, we constructed  a chain--collider-fork decomposition of $TT_n$ with $\frac{n}{4}(n-2)$ forks when $n$ is even and $\frac{(n-1)^2}{4}$ when $n$ is odd. Thus, the above upper bounds are tight, proving the claim.     

\end{proof}

\begin{note}
The construction above shows that the forks are chosen optimally by selecting all possible pairs of dots in each row of the dots-in-cells diagram. Therefore, for any $n$, the fork-packing number of $TT_n$ is
$n(n-2)/4$ when $n$ is even; $(n-1)^2/4$ when $n$ is odd.
\end{note}

\section{Conclusion}
\label{S:sec7}

In this paper, we have investigated decompositions and packings of transitive tournaments using the three connected two-arc motifs: chains, colliders, and forks. First, we constructed decompositions of \(TT_n\) into a mixture of these motifs whenever such decompositions are admissible. We then considered the corresponding pure packing problem for each individual motif and determined the maximum number of arc-disjoint copies of a chain, a collider, or a fork in \(TT_n\).


Let \(H\) denote a chain, a collider, or a fork. Then the \(H\)-packing number of \(TT_n\) for any $n$ is
\[
P(H,TT_n)=
\begin{cases}
\displaystyle
\frac{n(n-2)}{4}
& \text{if \(n\) is even},\\[1ex]
\displaystyle
\frac{(n-1)^2}{4}
& \text{if \(n\) is odd}.
\end{cases}
\]

A transitive tournament is a special class of directed acyclic graph (DAG). DAGs are widely used in causal modelling, including in epidemiology, where they provide a graphical framework for representing a priori assumptions about cause-and-effect relationships and the underlying data-generating process \cite{Piccininni2020}. Chains, colliders, and forks are fundamental local configurations in DAGs \cite{ByeonLee2023}. 

Although the present paper focuses on the highly structured setting of transitive tournaments, the results give a precise extremal understanding of how these basic two-arc motifs can be packed in an acyclic directed graph. This suggests possible future directions involving motif decompositions and packing problems in broader classes of DAGs, with potential relevance to the structural analysis of causal graphs.


\begin{thebibliography}{9}

\bibitem{AlspachPullman1974}
    B. R. Alspach and N. J. Pullman, 
    Path decompositions of digraphs,
    {\it Bull. Aust. Math. Soc.,} {\bf 10 (3)} (1974), 421-427.

\bibitem{AlspachMasonPullman1976}
    B. Alspach, D. Mason and N. Pullman, 
    Path numbers of tournaments,
    {\it J. Combin. Theory Ser. B,} {\bf 20} (1976), 222-228.


\bibitem{BangJensenGutin2009}
    J. Bang-Jensen,
    and G. Z. Gutin,
    {\it Digraphs: Theory, Algorithms and Applications},
    2nd ed., Springer Monographs in Mathematics,
    Springer, London, 2009.

\bibitem{ByeonLee2023}
S. Byeon and W. Lee,
Directed acyclic graphs for clinical research: a tutorial,
{\it J. Minim. Invasive Surg.} {\bf 26 (3)} (2023), 97--107.

\bibitem{CaroSchonheim1980}
Y. Caro and J. Sch\"onheim,
Decomposition of trees into isomorphic subtrees,
{\it Ars Combin.,} {\bf 9} (1980), 119--130.
    

\bibitem{ChuFanZhou2021}
    Y. Chu, G. Fan, and C. Zhou,
    Decompositions of 6-regular bipartite graphs into paths of length six,
    {\it Graphs Combin.} {\bf 37 (1)} (2021), 263--269.

\bibitem{DeVasDevillers2024}
    A. De Vas Gunasekara and A. Devillers,
    Transitive path decompositions of Cartesian products of complete graphs,
    {\it Des. Codes Cryptogr.,} {\bf 92} (2024), 4231--4245.
    

\bibitem{GiraoGranetKuhnLoOsthus2023}
    A. Girão, B. Granet, D. Kühn, A. Lo, and D. Osthus,
    Path decompositions of tournaments,
    {\it Proc. Lond. Math. Soc.} {\bf 126 (3)} (2023), no. 2, 429--517.


\bibitem{Gorlichetal2007}
    A. G{\"o}rlich, R. Kalinowski, M. Meszka, M. Pil\'sniak and M. Wo\'zniak,
    \textit{A note on decompositions of transitive tournaments},
    Discrete Mathematics, 307(7--8):896--904, 2007.
    doi:10.1016/j.disc.2005.11.045

\bibitem{Gorlichetal2007II}
    A. G{\"o}rlich, R. Kalinowski, M. Meszka and  M. Pil\'sniak, 
    \textit{A note on decompositions of transitive tournaments II},
    Australasian Journal of Combinatorics, \textbf{37} (2007), 57--66.


\bibitem{LoPatelSkokanTalbot2020}
    A. Lo, V. Patel, J. Skokan, and J. Talbot, 
    \textit{Decomposing tournaments into paths},
    Proc. Lond. Math. Soc., \textbf{121 (2)} (2020), 426--461.

\bibitem{Piccininni2020}
M. Piccininni, S. Konigorski, J. L. Rohmann and T. Kurth,
Directed acyclic graphs and causal thinking in clinical risk prediction modeling,
{\it BMC Med. Res. Methodol.} {\bf 20} (2020), Article 179.

\bibitem{SaliSimonyiTardos2021}
A. Sali, G. Simonyi, and G. Tardos,
Partitioning transitive tournaments into isomorphic digraphs,
{\it Order,} {\bf 38} (2021), 211--228.

\bibitem{Tarsi1983}
    M. Tarsi,
    Decomposition of a complete multigraph into simple paths: non-balanced handcuffed designs,
    {\it J. Combin. Theory Ser. A,} {\bf 34 (1)} (1983), 60--70. 


\bibitem{WangWuAn2017}
F. Wang, B. Wu, and X. An,
A \( \overrightarrow{P_3} \)-decomposition of tournaments and bipartite digraphs,
{\it Discrete Appl. Math.,} {\bf 226} (2017), 158--165.

\end{thebibliography}
\end{document}